\documentclass[11pt,reqno]%
{amsart}
\usepackage{verbatim}
\usepackage{amssymb}
\usepackage{enumerate}
\usepackage[active]{srcltx}

\newtheorem{theorem}{Theorem}
\newtheorem{proposition}[theorem]{Proposition}
\newtheorem{lemma}[theorem]{Lemma}
\newtheorem{corollary}[theorem]{Corollary}
\newtheorem{claim}[theorem]{Claim}

{}

\newtheorem{qtheorem}{Koszmider's Theorem}

\theoremstyle{definition}

\newtheorem{definition}[theorem]{Definition}

\newcommand{\al}{\alpha}
\newcommand{\be}{\beta}
\newcommand{\ga}{\gamma}
\newcommand{\omg}{\omega}
\newcommand{\ka}{\kappa}
\newcommand{\la}{\lambda}

\newcommand{\ccal}{{\mathcal C}}
\newcommand{\ical}{{\mathcal I}}
\newcommand{\pcal}{{\mathcal P}}

\newcommand{\setm}{\setminus}
\newcommand{\empt}{\emptyset}
\newcommand{\subs}{\subset}

\newcommand{\tip}{\operatorname{tp}}
\def\<{\left\langle}
\def\>{\right\rangle}
\def\cf{\operatorname{cf}}
\def\br#1;#2;{\bigl[ {#1} \bigr]^ {#2} }

\newif\ifdeveloping

\ifdeveloping
\usepackage[notref,notcite]{showkeys}
\fi
\usepackage{enumerate}

\newcommand{\EE}{\operatorname{E}}
\newcommand{\nnn}{\operatorname{n}}
\newcommand{\iii}{\operatorname{i}}
\newcommand{\II}{\operatorname{I}}

\newcommand{\concat}{\mathop{{}^{\frown}\makebox[-3pt]{}}}

\newcommand{\un}{\bigcup}

\newcommand{\mc}[1]{\mathcal{#1}}
\newcommand{\mbb}[1]{\mathbb{#1}}
\newcommand{\mb}[1]{\mathbf{#1}}
\newcommand{\mf}[1]{\mathfrak{#1}}

\newcommand{\preq}{\preceq_q}
\newcommand{\ip}{\iii_p}
\newcommand{\iq}{\iii_q}

\title[Superatomic Boolean Algebras]{
Superatomic Boolean algebras constructed from strongly unbounded
functions }

\author[J.C. Martinez]{Juan Carlos Mart\'{\i}nez}
\address{Facultat de Matem\`atiques \\ Universitat de Barcelona \\ Gran
  Via 585 \\ 08007 Barcelona, Spain}
\email{jcmartinez@ub.edu}
\thanks{The first author was supported by
the Spanish Ministry of Education DGI grant MTM2008-01545 and by
the Catalan DURSI grant 2009SGR00187. }

\author[L. Soukup]{Lajos Soukup }
\address{Alfr{\'e}d R{\'e}nyi Institute of Mathematics, Hungarian
  Academy of Sciences}
\email{soukup@renyi.hu}
\thanks{The second author was partially supported by
Hungarian National Foundation for Scientific Research grants no.
61600 and 68262}

\date{\today}

\begin{document}


\begin{abstract} Using Koszmider's strongly unbounded
functions, we show the following consistency result:

Suppose that $\ka,\la$ are infinite cardinals such that $\ka^{+++}
\leq \la$, $\ka^{<\ka}=\ka$  and $2^{\ka}= \ka^+$,
and $\eta$ is an ordinal with $\ka^+\leq \eta < \ka^{++}$ and
$\mbox{ cf}(\eta) = \ka^+$. Then, in some cardinal-preserving
generic extension there is
 a superatomic Boolean algebra
$\mbb B$ such that $\mbox{ht}(\mbb B) = \eta + 1$,
$\mbox{wd}_{\al}(\mbb B) = \ka$
 for every $\al < \eta$ and $\mbox{wd}_{\eta}(\mbb B) = \la$
(i.e.  there is a locally compact scattered space with cardinal
sequence $\<\kappa\>_\eta\concat \<\lambda\>$).

Especially, $\<{\omega}\>_{{\omega}_1}\concat \<{\omega}_3\>$ and
$\<{\omega}_1\>_{{\omega}_2}\concat \<{\omega}_4\>$ can be
cardinal sequences of  superatomic Boolean algebras.
\end{abstract}

\maketitle

\section{Introduction}
A Boolean algebra $\mbb B$ is {\em superatomic} iff every
homomorphic image of $\mbb B$ is atomic. Under Stone duality,
homomorphic images of a Boolean algebra $\mbb A$ correspond to
closed subspaces of its Stone  space $S(\mbb A)$, and atoms of
$\mbb A$  correspond to isolated points of $S(\mbb A)$. Thus $\mbb
B$ is superatomic iff its dual space $S(\mbb B)$ is {\em
scattered}, i.e. every non-empty (closed) subspace has some
isolate point.

For every Boolean algebra $\mbb A$, let $\mc I(\mbb A)$ be the
ideal
 generated by the atoms of  $\mbb A$.
Define, by induction on ${\alpha}$, the {\em ${\alpha}^{th}$
Cantor-Bendixson ideal $\mc J_{\alpha}(\mbb A)$}, and the {\em
${\alpha}^{th}$ Cantor-Bendixson derivative  $\mbb
A^{({\alpha})}$}  of $\mbb A$ as follows. If $\mc J_{\alpha}(\mbb
A)$ has been defined, put $\mbb A^{({\alpha})}=\mbb A/\mc
J_{\alpha}(\mbb A)$ and let  ${\pi}_{\alpha}:\mbb A\to \mbb
A^{({\alpha)}}$  be the canonical map. Define $\mc J_0(\mbb
A)=\{0_{\mbb A}\}$, $\mc J_{{\alpha}+1}(\mbb
A)={\pi}_{\alpha}^{-1}[\mc I(\mbb A^{{(\alpha})})]$, and for
${\alpha}$ limit $\mc J_{\alpha}(\mbb A)=\bigcup\{\mc
J_{\alpha'}(\mbb A):{\alpha'}<{\alpha}\}$. It is easy to see that
the sequence of the ideals $\mc J_{\alpha}(\mbb A)$ is increasing.
And it is a well-known fact that a non-trivial Boolean algebra
$\mbb A$ is superatomic iff there is an ordinal $\al$ such that
$\mbb A=\mc J_{\alpha}(\mbb A)$ (see \cite[Proposition 17.8]{Ko}).

 Assume that $\mbb B$ is a superatomic Boolean algebra.
The {\em height} of $\mbb B$, $ht(\mbb B)$, is the
 least ordinal
${\delta}$ such that $\mbb B=\mc J_{\delta}(\mbb B)$. This ordinal
$\delta$ is always a successor ordinal. Then, we define the {\em
reduced height} of $B$, $ht^{-}(\mbb B)$, as the least ordinal
${\delta}$ such that $\mbb B=\mc J_{\delta + 1}(\mbb B)$. It is
well-known that if $ht^{-}(\mbb B) = \delta$, then $\mc J_{\delta
+ 1}(\mbb B)\setminus \mc J_{\delta}(\mbb B)$ is a finite set. For
each ${\alpha}<ht^{-}(\mbb
 B)$ let $wd_{\alpha}(\mbb B)=|\mc J_{{\alpha}+1}(\mbb B)\setminus \mc J_{\alpha}(\mbb B)|$,
the number
 of atoms in $\mbb B/\mc J_{\alpha}(\mbb B)$.
The {\em cardinal sequence} of $\mbb B$, $CS(\mbb B)$, is the
sequence $\<wd_{\alpha}(\mbb B):{\alpha}<ht^-(\mbb B)\>$.

Let us turn now our attention from Boolean algebras to topological
spaces for a moment. Given a scattered space $X$, define, by
induction on ${\alpha}$, the {\em ${\alpha}^{th}$ Cantor-Bendixson
derivative  $X^{\alpha}$} of $X$ as follows: $X^0=X$,
$X^{\alpha}=\bigcap_{{\beta}<{\alpha}}X^{\beta}$ for limit
${\alpha}$, and $X^{\alpha+1}=X^{\alpha}\setm I(X^{\alpha})$,
where $I(Y)$ denotes the set of  isolated points of a space $Y$.
The set $I_{\alpha}(X)=X^{\alpha}\setm X^{{\alpha}+1}$ is the {\em
${\alpha}^{th}$ Cantor-Bendixson level} of $X$. The {\em reduced
height} of $X$, $ht^-(X)$, is the least ordinal ${\delta}$ such
that $X^{\delta}$ is finite (and so $X^{{\delta}+1}=\empt$). For
${\alpha}<ht^-(X)$ let $wd_{\alpha}(X)=|I_{\alpha}(X)|$. The {\em
cardinal sequence} of $X$, $CS(X)$, is defined as
$\<wd_{\alpha}(X):{\alpha}<ht^-(X)\>$.

It is well-known that if  $\mbb B$ is a superatomic Boolean
algebra, then the dual space of $\mbb B^{({\alpha})}$ is $(S(\mbb
B))^{({\alpha})}$ (see \cite[Construction 17.7]{Ko}). So
 $ht^-(\mbb B)=ht^-(S(\mbb B))$, and
$wd_{\alpha}(\mbb B)=wd_{\alpha}(S(\mbb B))$ for each $\al <
ht^-(\mbb B)$, that is,
 $\mbb B$ and $S(\mbb B)$
have the same cardinal sequences.

In this paper we consider   the following problem:  given a
sequence $\mb s$ of infinite cardinals, construct  a superatomic
Boolean algebra having $\mb s$ as its cardinal sequence.

For basic facts and results on superatomic Boolean algebras and
cardinal sequences we refer the reader to \cite{Ko} and \cite{R2}.
We shall use the notation $\<\kappa\>_{\alpha}$ to denote the
constant ${\kappa}$-valued sequence of length ${\alpha}$. Let us
denote the concatenation of two sequences $f$  and $g$  by $f
\concat g$. If ${\eta}$ is an ordinal we denote by $\mc C({\eta})$
the family of all cardinal sequences of superatomic Boolean
algebras whose reduced height is ${\eta}$.

If $\ka,\la$ are infinite cardinals and $\eta$ is an ordinal, we
say that a superatomic Boolean algebra $\mbb B$ is a {\em
$(\ka,\eta,\la)$-Boolean algebra} iff {$CS(\mbb
B)=\<\ka\>_\eta\concat \<\la\>$}, i.e.
 if $\mbox{ht}(\mbb B) = \eta + 1$,
$\mbox{wd}_{\al}(\mbb B) = \ka$ for each  $\al < \eta$ and
$\mbox{wd}_{\eta}(\mbb B) = \la$. An
$(\omega,\omega_1,\omega_2)$-Boolean algebra is called a {\em very
thin-thick Boolean algebra}. And, for an infinite cardinal $\ka$,
a $(\ka,\ka^+,\ka^{++})$-Boolean algebra is called a $\ka$-{\em
very thin-thick Boolean algebra}.

By using the combinatorial notion of the {\em new $\Delta$
property (NDP)} of a function, it was proved by Roitman  that the
existence of an $(\omg,\omg_1,\omg_2)$-Boolean algebra  is
consistent with ZFC (see \cite{R1} and \cite{R2}). It is worth to
mention that \cite{R1} was the first paper in which such a special
function was used to guarantee the chain condition of a certain
poset. Roitman's result was generalized in \cite{KM}, where for
every infinite regular cardinal $\ka$, it was proved that the
existence of a $(\ka,\ka^+,\ka^{++})$-Boolean algebra is
consistent with ZFC. Then, our aim here is to prove the following
stronger result.

\begin{theorem}\label{Theorem_1}
 Assume that $\ka,\la$ are
infinite cardinals such that $\ka^{+++} \leq \la$,
$\ka^{<\ka}=\ka$  and $2^{\ka}= \ka^+$. Then for each ordinal
$\eta$  with $\ka^+\leq \eta < \ka^{++}$ and $\cf(\eta) = \ka^+$,
in some  cardinal-preserving generic extension there is a
$(\ka,\eta,\la)$-Boolean algebra, i.e. $\<\ka\>_\eta\concat
\<\la\>\in \ccal(\eta+1)$.
\end{theorem}

\begin{corollary}
The existence of an $({\omega},{\omega}_1,{\omega}_3)$-Boolean
algebra is consistent with ZFC. An
$({\omega}_1,{\omega}_2,{\omega}_4)$-Boolean algebra may also
exist.
\end{corollary}

In order to prove 
Theorem \ref{Theorem_1}, we shall use the main result of \cite{K}.
Assume that $\ka,\la$ are infinite cardinals such that $\ka$ is
regular and $\ka <\la$. We say that a function $F :
[\la]^2\rightarrow \ka^+$ is a $\ka^+$-{\em strongly unbounded
function on} $\la$ iff for every ordinal $\delta < \ka^+$, every
cardinal $\nu < \ka$ and every family $A\subseteq [\la]^{\nu}$ of
pairwise disjoint sets with $|A|=\ka^+$,
 there are different $a,b\in A$
such that $F\{\al,\be\}> \delta$ for every $\al\in a$ and $ \be\in
b$. The following result was proved in  \cite{K}.
\begin{qtheorem}
If $\ka,\la$ are infinite cardinals such that $\ka^{+++} \leq
\la$, $\ka^{<\ka}=\ka$ and $2^{\ka}= \ka^+$, then there is a
$\ka$-closed and cardinal-preserving partial order that forces the
existence of a $\ka^+$- strongly unbounded function on $\la$.
\end{qtheorem}

So, in order to prove Theorem \ref{Theorem_1} it is enough to show
the following result.

\begin{theorem}\label{Theorem_2}
Assume that $\ka,\la$ are infinite cardinals with  $\ka^{+++} \leq
\la$ and $\ka^{<\ka}=\ka$, and  $\eta$ is an ordinal with
$\ka^+\leq \eta < \ka^{++}$ and $\cf(\eta) = \ka^+$. Assume that
there is a $\ka^+$- strongly unbounded function on $\la$. Then,
there is a cardinal-preserving partial order that forces the
existence of a $(\ka,\eta,\la)$-Boolean algebra.
\end{theorem}

In \cite{KM}, \cite{M}, \cite{R1} and in many other papers, the
authors proved the existence of certain superatomic Boolean
algebras in such a way that instead of constructing the algebras
directly, they actually produced certain ``graded posets'' which
guaranteed the existence of the wanted  superatomic Boolean
algebras. From these constructions,  Bagaria, \cite{Ba}, extracted
the following notion  and proved the Lemma \ref{lm:Ba} below which
was implicitly used in many earlier papers.

\begin{definition}[{\cite{Ba}}]\label{df:s-poset}
 Given a sequence $\mf s=\<{\kappa}_{\alpha}:{\alpha}<{\delta}\>$ of
 infinite cardinals, we say that a poset $\<T,\prec\>$ is an $\mf
 s$-{\em poset}
iff the following conditions are satisfied:

\begin{enumerate}[(1)]
\item $T=\bigcup\{T_{\alpha}:{\alpha}<{\delta}\}$ where
$T_{\alpha}=\{{\alpha}\}\times {\kappa}_{\alpha}$ for each $\al
<\delta$. \item For each $s\in T_{\alpha}$ and $t\in T_{\beta}$,
if $s \prec t$
  then ${\alpha}<{\beta}$.

\item For every $\{s,t\}\in \br T;2;$ there is a finite subset
$\iii\{s,t\}$ of $T$ such that for each $u\in T$:

 \begin{equation}\notag
  (u\preceq s\land u\preceq t) \text{ iff }
 \text{$u\preceq v$ for some $v\in \iii\{s,t\}$}.
  \end{equation}

\item For ${\alpha}<{\beta}<{\delta}$, if $t\in T_{\beta}$ then
the set $\{s\in T_{\alpha}:s\prec t\}$ is infinite.
\end{enumerate}
\end{definition}

\begin{lemma}[{\cite[Lemma 1]{Ba}}]\label{lm:Ba}
If there is an $\mf s$-poset then there is a superatomic Boolean
algebra with cardinal sequence $\mf s$.
\end{lemma}

Actually, if  $\mc T=\<T,\prec\>$ is an $\mf s$-poset, we write
$U_{\mc T}(x)=\{y\in T:y\preceq x\}$ for $x\in T$, and  we denote
by $X_{\mc T}$ the topological space on $T$ whose subbase is the
family
\begin{equation}
 \{U_{\mc T}(x), T\setm U_{\mc T}(x):x\in T\},
\end{equation}
then $X_{\mc T}$ is a locally compact, Hausdorff, scattered space
whose  cardinal sequence is $\mf s$, and so the clopen algebra of
the one-point compactification of $X_{\mc T}$ is the required
superatomic Boolean algebra with cardinal sequence $\mf s$.

\vspace{1mm} So,  to prove Theorem \ref{Theorem_2} it will be
enough to show that $\<{\kappa}\>_{\eta}\concat
\<{\lambda}\>$-posets may exist for $\kappa,\eta$ and $\lambda$ as
above.

\vspace{2mm} The organization of this paper is as follows. In
Section 2, we shall prove Theorem \ref{Theorem_2} for the special
case in which $\kappa = \omg$ 
and $\lambda \geq \omega_3$, generalizing in this way the result
proved by Roitman in \cite{R1}. In Section 3, we shall define the
combinatorial notions that make the proof of Theorem
\ref{Theorem_2} work. And in Section 4, we shall present the proof
of Theorem \ref{Theorem_2}.

\section{Generalization of Roitman's Theorem}

\normalsize

In this section, our aim is to prove the following result.

\begin{theorem}\label{Theorem_3} Let $\lambda$ be a cardinal with $\lambda \geq
\omega_3$. Assume that there is an $\omega_1$-strongly unbounded
function on $\la$. Then, in some  cardinal-preserving generic
extension for each ordinal $\eta$ with ${\omega}_1\le
{\eta}<{\omega}_2$ and $\cf({\eta})={\omega}_1$ there is an
$(\omega,{\eta},\la)$-Boolean algebra.
\end{theorem}

The theorem above is a bit stronger than Theorem \ref{Theorem_2}
for ${\kappa}={\omega}$, because the generic extension does not
depend on ${\eta}$. However, as we will see, its proof is much
simpler than the proof of the general case.

By Lemma \ref{lm:Ba}, it is enough to construct a c.c.c. poset
$\mc P$ such that in $V^{\mc P}$ for each $ {\eta}<{\omega}_2$
with $\cf({\eta})={\omega}_1$ there is an
$\<{\omega}\>_{\eta}\concat \<{\lambda}\>$-poset.

For ${\eta}={\omega}_1$ it is straightforward to obtain a suitable
$\mc P$: all we need is to plug  Kosmider's strongly unbounded
function into the original argument of Roitman. For
${\omega}_1<{\eta}<{\omega}_2$ this simple approach does not work,
but we can use the ``stepping-up'' method of Er-rhaimini and
Veli{\v c}kovic from \cite{EV}. Using this method, it will be
enough to construct a single $\<{\omega}\>_{{\omega}_1}\concat
\<{\lambda}\>$-poset (with some extra properties) to obtain
$\<{\omega}\>_{\eta}\concat \<{\lambda}\>$-posets for each $
{\eta}<{\omega}_2$ with $\cf({\eta})={\omega}_1$.

\vspace{1mm} To start with, we adapt the notion of a skeleton
introduced in {\cite{EV}} to the cardinal sequences we are
considering.
\begin{definition}
Assume  that  $\mc T=\<T,\prec\>$ is an $\mf s$-poset such that
$\mf s$ is a cardinal sequence of the form
$\<{\kappa}\>_{\mu}\concat \<{\lambda}\>$ where $\kappa, \lambda$
are infinite cardinals with $\kappa < \lambda$ and $\mu$ is a
non-zero ordinal. Let $i$ be the infimum function associated with
$\mc T$. Then
 for $\gamma < \mu$ we say that $T_{\gamma}$, the $\gamma^{th}$-level of $\mc T$, is a {\em
bone level}   iff the following holds:
\begin{enumerate}[(1)]

\item $i\{s,t\}=\empt$ for every $s,t\in T_{\gamma}$ with  $s\ne
t$. \item If $x\in T_{{\gamma}+1}$ and $y\prec x$ then there is a
$z\in T_{\gamma}$ with $y\preceq z\prec x $.

\end{enumerate}
We say that $\mc T$ is a {\em ${\mu}$-skeleton} iff $T_{\gamma}$
is a bone level of $\mc T$ for each ${\gamma}<{\mu}$.
\end{definition}

The next statement  can be proved by a straightforward
modification of the proof of \cite[Theorem 2.8]{EV}.

\begin{theorem}\label{tm:stepping_up}
Let  ${\kappa}, {\lambda}$ be infinite cardinals. If there is a
$\<{\kappa}\>_{{\kappa}^+}\concat\<{\lambda}\>$-poset which is a
${\kappa}^+$-skeleton, then for each ${\eta}<{\kappa}^{++}$ with
$cf({\eta})={\kappa}^+$ there is a
$\<{\kappa}\>_{\eta}\concat\<{\lambda}\>$-poset.
\end{theorem}

So, to get  Theorem \ref{Theorem_3} it is enough to prove the
following result.
\begin{theorem}\label{Theorem_3x} Let $\lambda$ be a cardinal with
$\lambda \geq \omega_3$. Assume that there is an $\omega_1$-
strongly unbounded function on $\la$. Then, in some  c.c.c.
generic extension there is an
$\<{\omega}\>_{{\omega}_1}\concat\<{\lambda}\>$-poset which is an
${\omega}_1$-skeleton.
\end{theorem}

Let  $F:[\lambda]^2\rightarrow \omega_1$ be an $\omega_1$-
strongly unbounded function on $\la$. In order to prove Theorem
\ref{Theorem_3x}, we shall define a c.c.c. forcing notion $\mc
P=\<P,\le\>$ that adjoins an $\mf s$-poset $\mc T=\<T,\preceq\>$
 which is an ${\omega}_1$-skeleton,
where $\mf s$ is the cardinal sequence
$\<{\omega}\>_{\omega_1}\concat \<{\lambda}\>$.

\vspace{1mm} So, the underlying set of the required $\mf s$-poset
 is the set $T = \un \{T_{\al}:\al \leq \omega_1 \}$ where
$T_{\al} = \{\al\}\times \omega$ for $\al <\omega_1$ and
$T_{\omega_1} = \{\omega_1\}\times \lambda$. If $s = (\al,\nu)\in
T$, we write $\pi(s) = \al$ and $\xi(s) = \nu$.

\vspace{2mm} Then, we define the poset $\pcal=\<P,\leq\>$ as
follows. We say that $p = \<X,\preceq,i\>\in P$ iff the following
conditions hold:
\begin{enumerate}[(P1)]
\item $X$ is a finite subset of $T$.
 \item $\preceq$ is a partial order on
$X$ such that $s \prec t$ implies $\pi(s)<\pi(t)$.

 \item  $i : [X]^2 \to [X]^{<\omega}$ is an infimum function, that is, a function such
 that for every
$\{s,t\}\in [X]^2$ we have:
  \begin{equation}\notag
  \forall x\in X ([x\preceq s\land x\preceq t]\text{ iff } x\preceq  v \text{ for some } v\in i\{s,t\}).
  \end{equation}

\item If $s,t\in X\cap T_{{\omega}_1}$ and $v \in i\{s,t\}$, then
$\pi(v) \in F\{\xi(s),\xi(t)\}$.

\item If $s,t\in X$ with ${\pi}(s)={\pi}(t)<{\omega}_1$, then
$i\{s,t\}=\empt$.

\item If $s,t\in X$, $s\prec t$ and ${\pi}(t)={\alpha}+1$, then
there is a $u\in X$ such that $s\preceq u\prec t$ and
${\pi}(u)={\alpha}$.

\end{enumerate}

\vspace{2mm} Now, we define $\leq$ as follows:
 $\<X',\preceq',i'\>\leq \<X,\preceq,i\> $
 iff $X\subseteq X'$, $\preceq\, =\,\preceq'\cap (X\times X)$ and
$i\subseteq i'$.

\vspace{2mm}  We will need condition (P4) in order to show that
$\pcal$ is c.c.c.

\begin{lemma}\label{Lemma 7}
 Assume that $p=\<X,\preceq,i\>\in  P$, $t\in X$,
$\alpha<\pi(t)$ and $n <\omega$. Then,  there is a
$p'=\<X',\preceq',i'\>\in P$ with $p'\leq p$ and there is an $s\in
X'\setminus X$ with $\pi(s) = \al$ and $\xi(s)> n$ such that, for
every $x\in X$, $s\preceq' x$ iff $t\preceq' x$. \end{lemma}

\begin{proof}
Let $L=\{{\alpha}\}\cup \{{\xi}: {\alpha} < {\xi} < {\pi}(t) \land
\exists j<{\omega}\ {\xi}+j= \pi(t)\}$. Let
${\alpha}={\alpha}_0,\dots, {\alpha}_\ell$ be the increasing
enumeration of $L$. Since $X$ is finite, we can pick an $s_j\in
T_{\al_j}\setminus X$ with $\xi(s_j)> n$
 for $j\le\ell$.
Let $X'=X\cup \{s_j:j\le\ell\}$ and let
$$
\prec'=\prec \cup \{(s_j,y): j\le l, t\preceq y\}\cup
\{(s_j,s_{k}):j<k\le\ell\} .
$$
Now, we put $i'\{x,y\} = i\{x,y\}$ if $x,y\in X$, $i'\{s_j,y\} =
\{s_j\}$ if $t\preceq y$, $i'\{s_j,s_k\}=s_{\min(j,k)}$, and
$i'\{s_j,y\} = \emptyset$ otherwise. Clearly, $\<X',\preceq',i'\>$
is as required.
\end{proof}

\begin{lemma}\label{Lemma 8} If $\pcal$ preserves cardinals,
then $\pcal$ adjoins an
$\<{\omega}\>_{{\omega}_1}\concat\<{\lambda}\>$-poset which is an
${\omega}_1$-skeleton.
\end{lemma}

\begin{proof}
Let $\mc G$ be a $\pcal$-generic filter. We put
$p=\<X_p,\preceq_p,i_p\>$ for $p\in \mc G$. By Lemma 10 and
standard density arguments, we have
\begin{equation}
T=\bigcup\{X_p:p\in \mc G\},
\end{equation}
and taking
\begin{equation}
\preceq=\bigcup\{\preceq_p:p\in \mc G\},
\end{equation}
the poset $\<T,\preceq\>$ is an $\<{\omega}\>_{\omega_1}\concat
\<{\lambda}\>$-poset. Especially, Lemma \ref{Lemma 7} ensures that
$\<T,\preceq\>$ satisfies (4) in Definition \ref{df:s-poset}.
 Properties (P5) and (P6) guarantee that
$\<T,\preceq\>$  is an ${\omega}_1$-skeleton.
  \end{proof}

Now, we prove the key lemma for showing that $\pcal$ adjoins the
required poset.

\begin{lemma}\label{Lemma9} $\pcal$ is c.c.c. \end{lemma}

\begin{proof}
Assume that $R = \<r_{\nu}:{\nu}<{{\omega_1}}\>\subseteq P$ with
$r_{\nu}\neq r_{\mu}$ for $\nu <\mu<\omega_1$. For $\nu
<\omega_1$, write $r_{\nu}=\<X_{\nu},\preceq_{\nu},\iii_{\nu}\>$
and put $L_{\nu} = \pi[X_{\nu}]$. By the $\Delta$-System Lemma, we
may suppose that the set $\{X_{\nu}:\nu <\omega_1\}$ forms a
$\Delta$-system with root $X^*$. By thinning out $R$ again if
necessary, we may assume that $\{L_{\nu}:\nu < \omega_1 \}$ forms
a $\Delta$-system with root $L^*$ in such a way that $X_{\nu}\cap
T_{\alpha} = X_{\mu}\cap T_{\alpha}$ for every $\al\in
L^*\setminus \{\omega_1\}$ and $\nu < \mu < \omega_1$. Without
loss of generality, we may assume that $\omega_1\in L^*$. Since
$\be\setminus \al$ is a countable set for $\al,\be\in L^*$ with
$\al < \be < \omega_1$, we may suppose that $ L^*\setminus
\{\omega_1\}$ is an initial segment of $L_{\nu}$ for every $\nu <
\omega_1$. Of course, this may require a further thinning out of
$R$. Now, we put $Z_{\nu} = X_{\nu}\cap T_{\omega_1}$ for $\nu <
\omega_1$. Without loss of generality, we may assume that the
domains of the forcing conditions of $R$ have the same size and
that there is a natural number $n > 0$ with $|Z_{\nu}\setminus
X^*| = |Z_{\mu}\setminus X^*| = n$ for $\nu <\mu <\omega_1$. We
consider in $T_{\omega_1}$ the well-order induced by $\lambda$.
Then, by thinning out $R$ again if necessary, we may assume that
for every $\{\nu,\mu\}\in [\omega_1]^2$ there is an
order-preserving bijection $h=h_{\nu,\mu} : L_{\nu}\rightarrow
L_{\mu}$ with $h\upharpoonright L^* = L^*$ that lifts to an
isomorphism of $X_{\nu}$ with $X_{\mu}$ satisfying the following:

\begin{enumerate}[(A)]
\item \vspace{1mm} For every $\al\in L_{\nu}\setminus
\{\omega_1\}$, $h(\al,\xi) = (h(\al),\xi)$.

\item  \vspace{1mm} $h$ is the identity on $X^*$.

\item  \vspace{1mm} For every $i < n$, if $x$ is the
$i^{th}$-element in $Z_{\nu}\setminus X^*$ and $y$ is the
$i^{th}$-element in $Z_{\mu}\setminus X^*$, then $h(x) = y$.

\item  \vspace{1mm} For every $x,y\in X_{\nu}$, $x\preceq_{\nu} y$
iff $h(x) \preceq_{\mu} h(y)$.

\item  \vspace{1mm} For every $\{x,y\}\in [X_{\nu}]^2$,
$h[i_{\nu}\{x,y\}] = i_\mu\{h(x),h(y)\}$.
\end{enumerate}

Now, we deduce from condition (P4) and the fact that $R$ is
uncountable that if $\{x,y\}\in [X^*]^2$ then
$i_{\nu}\{x,y\}\subseteq X^*$ for every $\nu < \omega_1$. So if
$\{x,y\}\in [X^*]^2$, then $i_{\nu}\{x,y\} = i_{\mu}\{x,y\}$ for
$\nu < \mu < \omega_1$.

\vspace{2mm} Let $\delta = \mbox{max}(L^*\setminus \{\omega_1\})$.
Since $F$ is an $\omega_1$-strongly unbounded function on
$\lambda$, there are ordinals $\nu,\mu$ with $\nu < \mu <\omega_1$
such that if we put $a =\{\xi\in\lambda : (\omega_1,\xi)\in
Z_{\nu}\setminus X^*\}$ and $a' = \{\xi\in\lambda :
(\omega_1,\xi)\in Z_{\mu}\setminus X^*\}$, then $F\{\xi,\xi'\} >
\delta$ for every $\xi\in a$ and every $\xi'\in a'$. Our purpose
is to prove that $r_{\nu}$ and $r_{\mu}$ are compatible in
$\pcal$.  We put $p= r_{\nu}$ and $q = r_{\mu}$. And we write
$p=\<X_p,\preceq_p,\ip\>$ and $q=\<X_q,\preq,\iq\>$. Then, we
define the extension $r = \<X_r,\preceq_r,i_r\>$ of $p$ and $q$ as
follows. We put $X_r = X_p\cup X_q$. We define $\preceq_r =
\preceq_p \cup \preceq_q$. Note that $\preceq_r$ is a partial
order on $X_r$, because $L^*\setminus \{\omega_1\}$ is an initial
segment of $\pi[X_p]$ and $\pi[X_q]$ . Now, we define the infimum
function $i_r$. Assume that $\{x,y\}\in [X_r]^2$. We put
$i_r\{x,y\} = i_p\{x,y\}$ if $x,y\in X_p$, and $i_r\{x,y\} =
i_q\{x,y\}$ if $x,y\in X_q$. Suppose that $x\in X_p\setminus X_q$
and $y\in X_q\setminus X_p$. Note that $x,y$ are not comparable in
$\<X_r,\preceq_r \>$ and there is no $u\in (X_p\cup X_q)\setminus
X^*$ such that $u \preceq_r x,y$. Then, we define $i_r\{x,y\} = \{
u\in X^*: u \prec_r x,y \}$. It is easy to check that $r\in P$,
and so $r\leq p,q$.
\end{proof}

After finishing the proof of  Theorem \ref{Theorem_2} for
${\kappa}={\omega}$, try to prove it for ${\kappa}={\omega}_1$.
So, assume that  $2^{\omega}={\omega}_1$, ${\omega}_4 \leq \la$,
and there is an ${\omega}_2$-strongly unbounded function on $\la$.
We want to find $\<{\omega}_1\>_{\eta}\concat
\<{\lambda}\>$-posets for each ordinal ${\eta}<{\omega}_3$ with
$\cf({\eta})={\omega}_2$ in some cardinal-preserving generic
extension. Since the ``stepping-up'' method of Er-rhaimini and
Veli{\v c}kovic worked for ${\kappa}={\omega}$, it is natural to
try to apply Theorem \ref{tm:stepping_up} for the case
${\kappa}={\omega}_1$. That is, we can try to  find a
cardinal-preserving generic extension that contains an
$\<{\omega}_1\>_{{\omega}_2}\concat \<\la\>$-poset which is an
${\omega}_2$-skeleton. For this, first we should consider the
forcing construction given in  \cite[Section 4]{KM} to add an
$\<{\omega}_1\>_{{\omega}_2}\concat \<\omega_3\>$-poset, and then
try to extend this construction to add the required
$\omega_2$-skeleton. However, the construction from \cite{KM} is
$\sigma$-complete and requires that CH holds in the ground model.
Then, the following results show that the forcing construction of
an $\<{\omega}_1\>_{{\omega}_2}\concat \<\la\>$-poset which is an
$\omega_2$-skeleton is quite hopeless, at least by using the
standard forcing from \cite{KM}.

If $X$ is the topological space associated  with a skeleton and
$x\in X$, we denote by $t(x,X)$ the tightness of $x$ in $X$.

\begin{proposition}\label{tm:problem}
Assume that $\mc T=\<T,\prec\>$ is a  $\mu$-skeleton, $\alpha <
\mu$ and $x\in I_{{\alpha}+1}(X_{\mc T})$. Then, $t(x,X_{\mc
T})={\omega}$.
\end{proposition}

\begin{proof}
Assume that $A\subseteq T$ and $x\in A'$. We can assume that
$a\prec x$ for each $a\in A$.

Let
\begin{equation}
U=\{u\in I_{\alpha}(X_{\mc T}):u\prec x\land \exists a_u\in A\
a_u\preceq u\}.
\end{equation}
Since $y\prec x$ iff $y\preceq u$ for some $u\prec x$ with $u\in
I_{\alpha}(X_{\mc T})$, the set $U$ is infinite.

Pick $V\in \br U;{\omega};$, and put $B=\{a_v:v\in  V\}$. We claim
that $x\in B'$. Indeed, if $y\prec x$ then there is a $u\in
I_{\alpha}(X_{\mc T})$ such that $y\preceq u\prec x $. So $|\{b\in
B: b\preceq y\}|\le 1$. Hence $y\notin B'.$ However, $B$ has an
accumulation point because $B\subseteq U_{\mc T}(x)$ and $U_{\mc
T}(x)$ is compact in $X_{\mc T}$.
So, $B$ should converge to $x$.
\end{proof}

\begin{corollary}
If $\mc T$ is a $\mu$-skeleton, then $\mu\le |I_0(X_{\mc
T})|^{\omega}$. Especially, under CH an
$\<{\omega}_1\>_{{\omega}_2}\concat \<\la\>$-poset  can not be an
$\omega_2$-skeleton.
\end{corollary}

Thus, we are unable to use Theorem \ref{tm:stepping_up} to prove
Theorem \ref{Theorem_2} even for  ${\kappa}={\omega}_1$. Instead
of this stepping-up method, in the next two sections we will
construct $\<{\omega}_1\>_{\eta}\concat \<{\lambda}\>$-posets
directly using the method of orbits from \cite{M}. This method was
used to construct by forcing $\<{\omega}_1\>_{{\eta}}$-posets for
$\omega_2\le \eta < \omega_3$. It is not difficult to get an
$\<{\omega}_1\>_{{\omega_2}}$-poset by means of countable
``approximations'' of the required poset. However, for $\omega_2\le
\eta < \omega_3$ we need the notion of orbit and a much more
involved forcing to obtain $\<{\omega}_1\>_{{\eta}}$-posets (see
\cite{M}).

\section{Combinatorial notions}

\normalsize

In this section, we define the combinatorial notions that will be
used in the proof of Theorem \ref{Theorem_2}.

 \vspace{2mm} If $\al,\be$ are ordinals with ${\alpha}\le {\beta}$ let
\begin{equation}
[{\alpha},{\beta})=\{{\gamma}:{\alpha}\le {\gamma}<{\beta}\}.
\end{equation}
We say that $I$ is an {\em ordinal interval} iff there are
ordinals ${\alpha}$ and ${\beta}$ with ${\alpha}\le {\beta}$ and
$I=[{\alpha},{\beta})$. Then, we write $I^-={\alpha}$ and
$I^+={\beta}$.

Assume that  $I=[{\alpha},{\beta})$ is an ordinal interval. If
${\beta}$ is a limit ordinal, let $\EE(I)=\{{\varepsilon}
^I_{\nu}:{\nu}<\mbox{ cf}(\beta)\}$ be a cofinal closed subset of
$I$ having order type $\cf{(\beta)}$ with ${\alpha} =
{\varepsilon}^I_{0}$, and then put
\begin{equation}
\operatorname{\mathcal
E}(I)=\{[{\varepsilon}^I_{\nu},{\varepsilon}^I_{\nu+1}):{\nu}<\cf{(\beta)}\}.
\end{equation}
If ${\beta}={\beta}'+1$ is a successor ordinal, put
 $\EE(I) =\{{\alpha},{\beta}'\}$
 and
\begin{equation}
\operatorname{\mathcal E}(I)=\{[{\alpha},{\beta}'),\{{\beta}'\}\}.
  \end{equation}

\vspace{1mm} Now, for an infinite cardinal $\kappa$ and an ordinal
$\eta$ with $\ka^+\leq \eta < \ka^{++}$ and $\mbox{ cf}(\eta) =
\ka^+$, we define ${\mathbb I}_{\eta} = \bigcup \{\ical_n : n <
\omega \}$ where:

\begin{equation}
\ical_0=\{[0,{\eta})\} \text{ and }
\ical_{n+1}=\bigcup\{\operatorname{\mathcal E}(I):I\in \ical_n\}.
 \end{equation}

 Note that ${\mathbb I}_{\eta}$ is a cofinal tree of intervals in the
sense defined in \cite{M}. So, the following conditions are
satisfied:

\begin{enumerate}[(i)]

\item  For every $I,J\in {\mathbb I}_{\eta}$, $I\subseteq J$ or
$J\subseteq I$ or $I\cap J = \emptyset$.

\item If $I,J$ are different elements of ${\mathbb I}_{\eta}$ with
$I\subseteq J$ and $J^+$ is a limit, \mbox{ then  $I^+ < J^+$ }.

 \item $\ical_n$ partitions $[0,\eta )$ for each $n < \omega$.

\item  $\ical_{n + 1}$ refines $\ical_n$ for each $n <\omega$.

\item  For every $\al <\eta$ there is an $I\in {\mathbb
I}_{\eta}$ such that $I^- = \al$.
\end{enumerate}

Then, for each ${\alpha} < {\eta}$ and $n<{\omega}$ we define $\II
( {\alpha},n)$ as the unique interval $I\in \ical_n$ such that
${\alpha}\in I$. And for each ${\alpha}< {\eta}$ we define
$n(\al)$ as the least natural number $n$ such that there is an
interval $I\in \ical_n$ with $I^-={\alpha}$. So if $n(\al) = k$,
then for every $m\geq k$ we have $I(\al,m)^-= \al$.

\vspace{2mm} Assume that ${\al}<{\eta}$. If $m < \nnn ({\al})$, we
define $o_m(\al) = \EE({\II( {\al},m)})\cap {\al}$. Then, we
define the {\em orbit} of $\al$ (with respect to ${\mathbb
I}_{\eta}$) as
\begin{equation}
o( {\al})=\bigcup\{o_m(\al): m<\nnn( {\al})\}.
\end{equation}

\vspace{1mm} For basic facts on orbits and trees of intervals, we
refer the reader to \cite[Section 1]{M}. In particular, we have
$|o(\al)|\leq \kappa$ for every ${\al}<{\eta}$.

We write $E([0,{\eta}))=\{{\varepsilon}_{\nu}:{\nu}<{\kappa}^+\}$.
  \begin{claim}\label{cl:epsilon}
$o({\varepsilon}_{\nu})=\{{\varepsilon}_{\zeta}:{\zeta}<{\nu}\}$
for ${\nu}<{\kappa}^+$.
  \end{claim}

  \begin{proof}
Clearly  $I({\varepsilon}_{\nu},0)=[0,{\eta})$
and $I({\varepsilon}_{\nu},1)=[{\varepsilon}_{{\nu}},
  {\varepsilon}_{{\nu}+1})$. So $n({\varepsilon}_{\nu})=1$.
Thus
$o({\varepsilon}_{\nu})=o_0({\varepsilon}_{\nu})=
E(I({\varepsilon}_{\nu},0))\cap {\varepsilon}_{\nu}=
E([0,{\eta}))\cap {\varepsilon}_{\nu}=\{{\varepsilon}_{\zeta}:{\zeta}<{\nu}\}$.
\end{proof}

For ${\alpha}<{\beta}<{\eta}$
let
\begin{equation}
j({\alpha},{\beta})=\max\{j: I({\alpha},j)=I({\beta},j)\},
\end{equation}
and put
\begin{equation}
J({\alpha},{\beta})=I({\alpha}, j({\alpha},{\beta})+1).
\end{equation}
For ${\alpha}<{\eta}$ let
\begin{equation}
J({\alpha},{\eta})=I({\alpha},1).
\end{equation}

\begin{claim}\label{cl:jplussz}
If  $\varepsilon_\zeta\le \al< \varepsilon_{\zeta+1}\leq \be\le
\eta$, then
$J(\al,\be)=[\varepsilon_\zeta,\varepsilon_{\zeta+1})$.
\end{claim}

\begin{proof}
For ${\beta}={\eta}$, $J(\al,\be)= I(\al,1) =
[\varepsilon_\zeta,\varepsilon_{\zeta+1})$.

Now assume that ${\beta}<{\eta}$. Since
$I(\al,0)=I(\be,0)=[0,\eta)$, but $I(\al,1)=[\varepsilon_\zeta,
\varepsilon_{\zeta+1} )$ and
$I(\be,1)=[\varepsilon_\xi,\varepsilon_{\xi+1})$ for some
$\varepsilon_\xi$ with  $\varepsilon_{\zeta+1}\le\varepsilon_\xi$,
we have
 $j(\al,\be)=0$ and so  $J(\al,\be)=[\varepsilon_\zeta,  \varepsilon_{\zeta+1} )$.
\end{proof}

\section{Proof of the Main Theorem}

In order to prove Theorem \ref{Theorem_2}, suppose that $\ka,\la$
are infinite cardinals with  $\ka^{+++} \leq \la$ and
$\ka^{<\ka}=\ka$, $\eta$ is an ordinal with $\ka^+\leq \eta <
\ka^{++}$ and $\mbox{cf}(\eta) = \ka^+$, and 
there is a $\ka^+$-
strongly unbounded function on $\la$. We will use a refinement of
the arguments given in \cite{M} and \cite[Section 4]{KM}.

\vspace{1mm} First, we define the underlying set of our
construction. For every ordinal $\al <\eta$, we put $T_{\al} =
\{\al\}\times \kappa$. And we put $T_{\eta} = \{\eta\}\times
\lambda$. We define $T = \un \{T_{\al}:\al \leq \eta \}$.
Let $T_{<\eta}=T\setm T_\eta$.
If $s =
(\al,\nu)\in T$, we write $\pi(s) = \al$ and $\xi(s) = \nu$.

\vspace{1mm} We put $\mathbb I = {\mathbb I}_{\eta}$. Also, we
define $E = E([0,\eta))=\{{\varepsilon}_{\nu}:{\nu}<{\kappa}^+\}$.
Since there is a  $\ka^+$-{strongly unbounded function on} $\la$
and $\cf({\eta})={\kappa}^+$ there is a function $F:\br
{\lambda};2;\to E$ such that
\begin{itemize}
 \item[$(\star)$]
For every ordinal ${\gamma}<{\eta}$ and every family $A\subseteq
[\la]^{<{\kappa}}$ of pairwise disjoint sets with $|A|=\ka^+$,
 there are different $a,b\in A$
such that $F\{\al,\be\}> {\gamma}$ for every $\al\in a$ and $ \be\in
b$.
\end{itemize}

\


\vspace{1mm} Let $\Lambda\in \mathbb I$ and $\{s,t\}\in [T]^2$
with $\pi(s) < \pi(t)$. We say that {\em $\Lambda$ isolates $s$
from $t$} iff $\Lambda^- <{\pi} (s)<\Lambda^+$ and $\Lambda^+\le
{\pi} (t)$.

\vspace{2mm} Now we define the poset
$\pcal=\<P,\leq\>$ as follows. We say that $p
= \<X,\preceq,i\>\in P$ iff the following conditions hold:
\begin{enumerate}[(P1)]
\item $X\in [T]^{<\kappa}$.
 \item $\preceq$ is a partial order on
$X$ such that $s \prec t$ implies ${\pi}(s)<{\pi}(t)$.

 \item  $\iii:\br X;2; \to X\cup\{{\rm undef}\}$ is an infimum function, that is, a function such
 that for every
$\{s,t\}\in \br X;2;$ we have:
  \begin{equation}\notag
  \forall x\in X ([x\preceq s\land x\preceq t]\text{ iff } x\preceq \iii\{s,t\}).
  \end{equation}

\item If $s,t\in X$ are compatible but not comparable in
$(X,\preceq)$, $v = i\{s,t\}$  and $\pi(s) = \al_1$, $\pi(t) =
\al_2$ and $\pi(v) = \be$, we have:

\begin{itemize}
\item[(a)] If $\al_1,\al_2 <\eta$, then $\be\in o(\al_1)\cap
o(\al_2)$.

\item[(b)] If $\al_1<\eta$ and $\al_2 = \eta$, then $\be\in
o(\al_1)\cap E$.

\item[(c)] If $\al_1 = \eta$ and $\al_2 < \eta$, then $\be\in
o(\al_2)\cap E$.

\item[(d)] If $\al_1 = \al_2 = \eta$, then $\be\in
F\{\xi(s),\xi(t)\}\cap E$.

\end{itemize}

\item If $s,t\in X$ with $s \preceq t$ and
$\Lambda=J({\pi}(s),{\pi}(t))$ isolates $s$ from $t$, then there
is a $u\in X$ such that $s \preceq u\preceq t$ and $\pi(u) =
{\Lambda}^+$.
\end{enumerate}

\vspace{2mm} Now, we define $\leq$ as follows:
 $\<X',\preceq',\iii'\>\leq\<X,\preceq,\iii\> $
 iff $X\subseteq X'$, $\preceq\, =\,\preceq'\cap (X\times X)$ and
$\iii\subseteq \iii'$.

 \begin{lemma}\label{Lemma_4}
 Assume that $p=\<X,\preceq,\iii\>\in  P$, $t\in X$,
${\alpha}<{\pi} (t)$ and $\nu <\kappa$. Then,  there is a
$p'=\<X',\preceq',\iii'\>\in P$ with $p'\leq p$ and
there is an $s\in X'\setminus X$ with $\pi(s) = \al$ and
$\xi(s)>\nu$ such that, for every $x\in X$, $s\preceq' x$ iff
$t\preceq' x$.
  \end{lemma}

\begin{proof} Since $|X|<\kappa$, we can
take an $s\in T_{\al}\setminus X$ with $\xi(s)>\nu$. Let
$\{I_0,\dots,I_n\}$ be the list of  all the intervals in $\mathbb
I$ that isolate $s$ from $t$ in such a way that $I^+_0
> I^+_1 > \dots >I^+_n$. Put $\gamma_i = I^+_i$ for $i\leq n$. We
take points $c_i\in T\setminus X$ with $\pi(c_i) = \gamma_i$ for
$i\leq n$. Let $X'=X\cup \{s\} \cup \{c_i: i\leq n\}$ and let
\begin{multline}\notag
\prec'=\prec\cup \{\<s,c_i\>: i\leq n\} \cup \{\<s,y\>: t\preceq
y\} \cup \{\<c_j,c_i\>: i < j\}
\\\cup\{\<c_i,y\>:i\leq n, t\preceq y\}.
\end{multline}

Note that, for $z\in X'$ and $y\in \{s\}\cup \{c_i: i\leq n \}$,
either $z$ and $y$ are comparable or they are incompatible with
respect to $\preceq'$. So, the definition of $i'$ is
clear.

Finally observe that  $p'$ satisfies (P5) because
if $x\prec' y$, $x\in \{s\}\cup \{c_i:i\le n\}$ and $y\in X'$ then
$J(\pi(x), \pi(y))=I_k$ for some $0\le k\le n$, so $c_k$ witnesses $(P5)$ for 
$x$ and $y$. 
\end{proof}

For $p\in P$ we write $p=\<X_p,\preceq_p, i_p\>$, $Y_p=X_p\cap
T_{<{\eta}}$ and $Z_p=X_p\cap T_{\eta}$.

\begin{lemma}\label{Lemma_5x} If $\pcal$ preserves
cardinals, then forcing with $\pcal$ adjoins a
$(\kappa,\eta,\lambda)$-Boolean algebra.
\end{lemma}

\begin{proof}
Let $\mc G$ be a $\pcal$-generic filter. Then
\begin{equation}
T=\bigcup\{X_p:p\in \mc G\},
\end{equation}
and taking
\begin{equation}
\preceq = \bigcup\{\preceq_p:p\in \mc G\}
\end{equation}
the poset $\<T,\preceq\>$ is a $\<{\kappa}\>_{\eta}\concat
\<{\lambda}\>$-poset. Especially, Lemma \ref{Lemma_4} guarantees that
$\<T,\prec\>$ satisfies (4) from Definition \ref{df:s-poset}. So, by Lemma
\ref{lm:Ba}, in $V[\mc G]$ there is a
$(\kappa,\eta,\lambda)$-Boolean algebra.
  \end{proof}

To complete our proof we should check that forcing with $P$
preserves cardinals. It is straightforward that $\pcal$ is
$\kappa$-closed. The burden of our proof is to verify the
following statement, which completes the proof of Theorem
\ref{Theorem_2}.
\begin{lemma}\label{Lemma_6}
 $\pcal$  has the $\kappa^+$-chain condition.
\end{lemma}

Define the subposet $\pcal_{\eta}=\<P_{\eta},\leq_{\eta}\>$ of
$\pcal$ as follows:
\begin{equation}
 P_{\eta}=\{p\in P: x_p\subseteq {\eta}\times {\kappa}\},
\end{equation}
and let $\le_{\eta}=\le\restriction P_{\eta}$. The poset
$\pcal_{\eta}$ was  defined in \cite[Definition 2.1]{M}, and it
was proved that $\pcal_{\eta}$ satisfies the ${\kappa}^+$-chain
condition. In \cite[Lemmas 2.5 and 2.6]{M} it was shown that every
set $R \in \br P_{\eta};{\kappa}^+;$ has a linked subset of size
${\kappa}^+$. Actually, a stronger statement was proved, and we
will use that statement to prove  Lemma \ref{Lemma_6}. However,
before doing so, we need some preparation.

\begin{definition}
Suppose that $g : A \to B$ is a bijection, where $A, B \in \br
T;<\kappa;$. We say that $g$ is {\em adequate} iff the following
conditions hold:
\begin{enumerate}[(1)]
\item     $g[A\cap T_{<\eta}]=B\cap T_{<\eta}$ and  $g[A\cap T_\eta]=B\cap T_\eta$.
 \item  For every $s, t \in A$, $\pi(s) < {\pi}(t)$ iff ${\pi}(g(s)) < {\pi}(g(t))$.
\item For every $s = \<\alpha, \nu\> \in  A\cap T_{<\eta}$,
$g(\alpha, \nu) = (\beta, ξ\zeta)$ implies $\nu = \zeta$. \item
For every $s,t\in A\cap T_\eta$, $\xi(s)<\xi(t)$ iff
$\xi(g(s))<\xi(g(t))$.
\end{enumerate}
\end{definition}
For  $A,B\subseteq T_{<\eta}$, this definition is just
{\cite[Definition 2.2]{M}}.

\begin{definition}
A set $Z \subseteq P$ is
{\em separated}  iff the following conditions are satisfied:
\begin{enumerate}[(1)]
 \item
 $\{X_p : p \in Z\}$ forms a $\Delta$-system with root $X$.
\item For each $\alpha<\eta$, either $X_p \cap T_\alpha = X \cap
T_\alpha$ for every $p \in Z$, or there is at most one $p \in  Z$
such that $X_p \cap  T_\alpha \neq \emptyset $. \item For every
$p, q \in  Z$ there is an adequate bijection $h_{p,q} : X_p \to
X_q$ which satisfies the following:
\begin{enumerate}[(a)]
 \item For any $s \in  X$, $h_{p,q} (s) = s$.
\item If $s, t \in  X_p$ , then $s \prec_p t$ iff $h_{p,q} (s) \prec_q h_{p,q} (t)$.
\item  If $s, t\in  X_p$, then
$h_{p,q} (\ip \{s, t\}) = \iq \{h_{p,q} (s), h_{p,q} (t)\}$.
\end{enumerate}
\end{enumerate}
\end{definition}
For  $Z\subseteq P_{\eta}$, this definition is just
{\cite[Definition 2.3]{M}}.


\begin{lemma}\label{lm:kerneldown}
Assume that $Z\in \br P;{\kappa^+}; $  is separated and $X$ is the
root of the $\Delta$-system $\{X_p : p\in Z \}$. If $s,t$ are
compatible but not comparable in $p\in Z$ and $s\in X\cap
T_{<{\eta}}$, then $\ip\{s,t\}\in X$.
\end{lemma}

\begin{proof} Assume that $s,t$ are compatible but not comparable in $p\in Z$ and
$s\in X\cap T_{<{\eta}}$. Assume that $\ip\{s,t\}\not\in X$. Then
since
\begin{equation}
\{\iq\{s,h_{p,q}(t)\}: q\in Z\}= \{h_{p,q}(\ip\{s,t\}): q\in Z\},
\end{equation}
the elements of $\{\iq\{s,h_{p,q}(t)\}: q\in Z\}$ are all
different. But this is impossible, because
$\pi(\iq\{s,h_{p,q}(t)\})\in o(s)$ for all $q\in Z$ and $|o(s)|\le
\kappa$.
\end{proof}

In \cite[Lemmas 2.5 and 2.6]{M}, as we explain in the Appendix of
this paper, actually the following statement was proved.

\begin{proposition}\label{proposition_23}
For each  subset $R\in \br P_{\eta};{{\kappa}^+};$ there is  a
separated subset $Z\in \br R;{{\kappa}^+};$ and an ordinal
${\gamma}<{\eta}$ such that every  $p,q\in Z$ have a common
extension $r\in P_\eta$ such that the following holds:

\begin{enumerate}[(R1)]
 \item $\sup \ {\pi}[X_r\setm(X_p\cup X_q)]<{\gamma}. $
 \item
\begin{enumerate}[(a)]
 \item
$y\prec_r s $ iff $y\prec_r h_{p,q}(s)$ for each $s\in X_p$ and
$y\in X_r\setm(X_p\cup X_q)$,\item $s\prec_r y $ iff
$h_{p,q}(s)\prec_r  y$ for each $s\in X_p$ and $y\in
X_r\setm(X_p\cup X_q)$, \item if $s\prec_r y $ for $s\in X_p\cup
X_q$ and $y\in X_r\setm(X_p\cup X_q)$, then there is a $w\in
X_p\cap X_q$ with $s\preceq_r w \prec_r y$, \item for $s\in
X_p\setm  X_q$ and $t\in X_q\setm X_p$,
\begin{gather}\label{g1}
\text{$s\prec_r t$ iff $\exists u\in X_p\cap X_q$ such that
$s\prec_p u\prec_q t$,}   \\\notag \text{$t\prec_r s$ iff $\exists
u\in X_p\cap X_q$ such that $t\prec_q u\prec_p s$.}
\end{gather}
\end{enumerate}

\end{enumerate}

\end{proposition}

After this preparation, we are ready to prove Lemma \ref{Lemma_6}.

\begin{proof}[Proof of Lemma \ref{Lemma_6}]
We will argue in the following way. Assume that $R =
\<r_{\nu}:{\nu}<{{\kappa^+}}\>\subseteq P$, where
$r_{\nu}=\<X_{\nu},\preceq_{\nu},\iii_{\nu}\>$. For each $\nu
<\ka^+$ we will ``push down'' $r_\nu$ into $P_\eta$, more
precisely, we will construct an isomorphic copy $r'_\nu\in P_\eta$
of $r_\nu$. Using Proposition \ref{proposition_23} we can find a separated
subfamily $\{r'_\nu:\nu\in K\}$ of size $\ka^+$ and  an ordinal
$\ga<\eta$ such that for each $\nu, \mu\in K$ with $\nu\ne \mu$
there is a condition $r'_{\nu,\mu}\in P_\eta$ such that
$r'_{\nu,\mu}\le_\eta r'_\nu,r'_\mu$ and (R1)--(R2) hold,
especially
\begin{equation}
 \sup {\pi}[X'_{{\nu,\mu}}\setm(X'_\nu\cup X'_\mu)]<{\gamma}.
\end{equation}
Let $X$ be the root of $\{X_{\nu}:\nu < \ka^+\}$, $Y=X\setminus
T_{\eta}$ and $\gamma_0=\max(\gamma, \mbox{ sup } \pi [Y])$. Since
$F$ is  $\ka^+$-strongly unbounded, there are $\nu, \mu\in K$ with
$\nu<\mu$ such that
\begin{equation}
 \forall s\in (X_\nu\setm X_\mu)\cap T_\eta\quad
\forall t\in (X_\mu\setm X_\nu)\cap T_\eta\quad
F\{\xi(s),\xi(t)\}>\gamma_0.
\end{equation}
Then we will be able to ``pull back'' $r'=r'_{\nu,\mu}$ into $P$
to get a condition $r=r_{\nu,\mu}$ which is a common extension of
$r_\nu$ and $r_\mu$. Let us remark that $r$ will not be an isomorphic copy of $r'$,
rather $r$ will be a ``homomorphic image`` of $r'$.

Now we carry out our plan.

Since ${\kappa}^{<{\kappa}}={\kappa}$, by thinning out our
sequence we can assume that $R$ itself is a separated set. So
$\{X_r:r\in R\}$ forms a $\Delta$-system with kernel $\bar X$. We
write $\bar Y=\bar X\cap T_{<\eta}$ and $\bar Z=\bar  X\cap
T_\eta$.

Recall that
$E=E([0,{\eta}))=\{{\varepsilon}_{\zeta}:{\zeta}<{\kappa}^+\}$
 is a closed unbounded subset of ${\eta}$.

Fix ${\nu}<{\kappa}^+$. Write $Y_{\nu}= X_{\nu}\cap T_{<\eta} $
and $Z_{\nu}= X_{\nu}\cap T_{\eta} $. Pick  a limit ordinal
${\zeta}({\nu})<{\kappa}^+$ such that:
\begin{enumerate}[(i)]
\item $\sup (\pi[Y_\nu])<{\varepsilon}_{{\zeta}({\nu})}$, \item
${\zeta}({\mu}) < {\zeta}({\nu})$ for $\mu < \nu$.
\end{enumerate}

\noindent Let $\theta = \tip(\xi[Z_{\nu}])$ and $\alpha =
\varepsilon_{\zeta(\nu)}$. We put $Z'_{\nu} = \{\<\alpha,\xi\> :
\xi < \theta \}$. Clearly, $Z'_{\nu}\subseteq
T_{\varepsilon_{\zeta(\nu)}}$ and
$\tip(\xi[Z'_{\nu}])=\tip(\xi[Z_{\nu}])$. We consider in
$Z'_{\nu}$ and $Z_{\nu}$ the well-orderings induced by $\kappa$
and $\lambda$ respectively. Put $X_{\nu}'=Y_{\nu}\cup Z'_{\nu}$,
and let $g_\nu:X'_{\nu}\to X_{\nu}$ be the natural bijection, i.e.
$g_\nu\restriction Y_{\nu}=id$ and $g_{\nu}(s) = t$ if for some
$\xi < \tip(\xi[Z_{\nu}])$ $s$ is the $\xi$-element in $Z'_{\nu}$
and $t$ is the $\xi$-element in $Z_{\nu}$.

 Let $\bar  Z'_\nu=g_\nu^{-1}\bar Z$. We define the
condition $r_{\nu}'=\<X_{\nu}',\preceq_{\nu}', \iii'_{\nu}\>\in
P_{\eta}$ as follows: for $s,t\in X'_{\nu}$ with $s\ne t$ we put

  \begin{equation}
    \text{  $s\prec'_{\nu} t$ iff $g_{\nu}(s)\prec_{\nu}
    g_{\nu}(t)$,
}
  \end{equation}
and
  \begin{equation}
    \text{  $\iii'_{\nu}\{s,t\}= \iii_{\nu}\{g_{\nu}(s), g_{\nu}(t)\}
$.
}
  \end{equation}

\medskip
\begin{claim} $r'_{\nu}\in P_{\eta}$. \end{claim}

\begin{proof}[Proof]
(P1), (P2) and (P3) are clear because
$g_{\nu}$ is an isomorphism between
$r'_\nu=\<X'_{\nu},\preceq'_{\nu},\iii'_{\nu}\>$
and $r_\nu=\<X_{\nu},\preceq_{\nu},\iii_{\nu}\>$,
moreover ${\pi}(s)<{\pi}(t)$
iff ${\pi}(g_{\nu}(s))<{\pi}(g_{\nu}(t))$.

\smallskip
\noindent (P4) Since $X'_\nu\subseteq T_{<\eta}$  we should check
just (a). So assume that $s',t'\in X'_{\nu}$ are compatible but
not comparable in $\<X'_{\nu},\leq'_{\nu}\>$ and
$v'=\iii'_{\nu}\{s',t'\}$. Put $s = g_{\nu}(s')$, $t =
g_{\nu}(t')$. Since $g_\nu\restriction Y_\nu=id$, we can assume
that $\{s',t'\}\notin \br Y_\nu;2;$, e.g. $s'\in Z'_\nu$ and so
$s\in Z_\nu$.

First observe that $v'\in  Y_{\nu}$, so $v'=g_\nu(v')$.

If $t'\in Y_{\nu}$, then $t'=g_\nu(t')$, and
$v'=\iii_{\nu}\{s,t'\}$.  By applying (P4)(c) in $r_{\nu}$ for $s$
and $t'$ we obtain
\begin{equation}
{\pi}(v')\in E\cap o({\pi}(t'))\subseteq E\cap
\varepsilon_{\zeta(\nu)}\cap o({\pi}(t'))=  o({\pi}(s'))\cap
o({\pi}(t'))
\end{equation}
because $o(\pi(s'))=E\cap \varepsilon_{\zeta(\nu)}$
by Claim \ref{cl:epsilon}.

If $t'\in Z'_{\nu}$, then $t=g_{\nu}(t')\in Z_{\nu}\subseteq
T_{\eta}$. Since $v' = i'_{\nu}\{s',t'\} = i_{\nu}\{s,t\}$,
applying (P4)(d) in $r_{\nu}$ for $s$ and $t$ we obtain
\begin{equation}\notag
{\pi}(v')\in F\{{\xi}(s), {\xi}(t)\}\cap E\cap
\varepsilon_{\zeta(\nu)} \subseteq E\cap \varepsilon_{\zeta(\nu)}
= o({\pi}(s'))\cap o({\pi}(t'))
\end{equation}
because $o(\pi(s'))=o(\pi(t'))=E\cap \varepsilon_{\zeta(\nu)}$
by Claim \ref{cl:epsilon}.

\smallskip
\noindent (P5) Assume that $s',t'\in X'_{\nu}$, $s'\prec'_\nu t'$
and $\Lambda=J({\pi}(s'), {\pi}(t'))$ isolates $s'$ from $t'$.
Then $s'\in Y_\nu$, so $g_{\nu}(s')=s'$. Since $g_\nu\restriction
Y_\nu=id$, we can assume that $\{s',t'\}\notin \br Y_\nu;2;$, i.e.
$t'\in Z'_\nu$.

Write $t=g_\nu(t')$. Since
$\pi(t')={\varepsilon}_{{\zeta}({\nu})}\in E$ , by Claim
\ref{cl:jplussz},
$J(\pi(s'),\pi(t'))=J(\pi(s'),\pi(t))=[\varepsilon_{\zeta},
\varepsilon_{\zeta+1} )=I({\pi}(s'),1)$, where
$\varepsilon_{\zeta}\le \pi(s')< \varepsilon_{\zeta+1}$.
Applying (P5) in $r_{\nu}$ for $s'$ and $t$ we obtain a $v\in
Y_{\nu}$ such that $\pi(v)=\Lambda^+$ and $s'\prec_{\nu}v
\prec_{\nu} t$. Then $g_{\nu}(v)=v$, so $s'\prec_{\nu}'
v\prec_{\nu}' t'$, which was to be proved.
\end{proof}

Now applying Proposition 23 to the family
$\{r'_{\nu}:{\nu}<{\kappa}^+\}$, there are $K\in \br {\kappa}^+;
{\kappa}^+;$ and ${\gamma}<{\eta}$ such that $\{r'_{\nu}:{\nu}\in
K\}$ is separated and  for every ${\nu}, {\mu}\in K$ with
${\nu}\ne{\mu}$ there is a common extension $r'\in P_{\eta}$ of
$r'_\nu$ and $r'_\mu$ such that (R1)-(R2) hold. Let
$\gamma_0=\max(\gamma, \mbox{ sup } \pi [\bar Y])$. Recall that
$\bar Y$ is the root of the $\Delta$-system $\{Y_\nu:\nu\in
\ka^+\}$. For $\nu < \mu < \kappa^+$ we denote by $h'_{\nu,\mu}$
the adequate bijection $h_{r'_{\nu},r'_{\mu}}$.

Since $F$ satisfies $(\star)$, there are ${\nu}, {\mu}\in K$ with
${\nu}\ne{\mu}$  such that for each $s\in (Z_{\nu}\setm Z_{\mu})$
and $t\in (Z_{\mu}\setm Z_{\nu})$ we have
\begin{equation}
F\{\xi(s),\xi(t)\}>{\gamma}_0.
\end{equation}

We show that the conditions $r_\nu$ and $r_\mu$ have a common
extension $r=\<X,\preceq,\iii\>\in P$.

Consider a condition $r'=\<X',\preceq',\iii'\>$ which is a common
extension of $r'_\nu$ and $r'_\mu$ and satisfies (R1)--(R2). We
define the condition $r = \<X,\preceq,\iii\>$  as follows. Let
\begin{equation}
X=(X'\setm (Z'_{\nu}\cup Z_{\mu}'))\cup (Z_{\nu}\cup Z_{\mu}).
\end{equation}
Write $U=X'\setm (Z'_{\nu}\cup Z_{\mu}') = X\setm (Z_{\nu}\cup
Z_{\mu}) $ and $V=X'\setm (X'_{\nu}\cup X_{\mu}')$. Clearly,
$V\subseteq U$. We define the function $h:X'\to X$ as follows:
\begin{equation}
 h=g_\nu\cup g_\mu \cup (id\restriction U).
\end{equation}
Then $h$ is well-defined, $h$ is onto, $h\restriction X'\setm
(\bar Z'_\nu\cup \bar Z'_\mu)$ is injective, and $h[\bar
Z'_\nu]=h[\bar Z'_\mu]=\bar Z$.

Now, if $s,t\in X$ we put

\begin{equation}\label{eq:prec}
s\prec t \mbox{ iff there is a } t'\in X' \mbox{ with } h(t')=t
\mbox{ and } s \prec' t'.
\end{equation}




Finally, we define the meet function $\iii$ on $\br X;2;$ as
follows:
\begin{equation}\label{eq:0}
\iii\{s,t\}=\max_{\prec'}\{\iii'\{s',t'\}: \text{  $h(s')=s$ and $h(t')=t$} \}.
\end{equation}
We will prove in the following claim that the definition of the
function  $\iii$ is meaningful. Then the proof of Lemma
\ref{Lemma_6} will be complete as soon as we verify  that $r\in P$
and $r\le r_\nu, r_\mu$.

\begin{claim}\label{cl:wd}
$\iii$ is well-defined by (\ref{eq:0}), moreover $\iii\supseteq
\iii_\nu\cup \iii_\mu$.
\end{claim}

\begin{proof}[Proof]
We need to verify that the maximum in (\ref{eq:0}) does exist when
we define $\iii\{s,t\}$. So, suppose that $\{s,t\}\in [X]^2$.

If $\{s,t\}\in \br X\setm {\bar Z};2;$ then there is exactly one
pair $(s',t')$ such that $h(s')=s$ and $h(t')=t$, and hence there
is no problem in (\ref{eq:0}). So if $\{s,t\}\in \br X_\nu;2;$
then $\iii\{s,t\}=\iii'\{s',t'\}= \iii_\nu\{s,t\}$ by the
construction of $r'_\nu$. If  $\{s,t\}\in \br X_\mu;2;$
proceeding similarly we obtain
$\iii\{s,t\}=\iii'\{s',t'\}=\iii_\mu\{s,t\}$.

So we can assume that e.g. $s\in {\bar Z}$. Then
$h^{-1}(s)=\{s',s''\}$ for some $s'\in {\bar Z}'_{\nu}$ and
$s''\in {\bar Z}'_\mu $.

First assume that  $t\notin {\bar Z}$, so there is exactly one
$t'\in X'$ with $h(t')=t$. We distinguish the following cases.

\medskip\noindent{\em Case 1.   $t\in V $}.

Note that since $t\in V$,  $t=t'$. We show that
$\iii'\{s',t\}=\iii'\{s'',t\}$.

Let $v=\iii'\{s',t\}$. Assume that $v\in X'_\nu\cup X'_\mu$. Then,
by (R2)(c), $v\prec' t$ and $t\in V$ imply that there is a $w\in
\bar Y=X'_\nu\cap X'_\mu$ such that $v\preceq' w \prec' t$. Thus
$v=\iii'\{s',w\}$ and $\iii'\{s',w\}=\iii'_\nu\{s',w\}
=\iii_\nu\{s,w\}\in \bar Y$ by Lemma \ref{lm:kerneldown} for $w\in
\bar Y$. Clearly, $v \prec' t,s''$. Hence $v\preceq'
\iii'\{s'',t\}$.

Now assume that $v\in V$. Then $v\prec' s'$ implies $v\prec'
h'_{\nu, \mu}(s')=s''$ by (R2)(a). So $v\prec' t,s''$, thus
$\iii'\{s',t\}\preceq'\iii'\{s'',t\}$.

So, in both cases $\iii'\{s',t\}\preceq'\iii'\{s'',t\}$.
But $s'$ and $s''$ are
symmetrical, hence $\iii'\{s'',t\}\preceq'\iii'\{s',t\}$, and  so
we are done.

\medskip\noindent{\em Case 2.   $t\in X_\nu\setm \bar Z $}.

We show that in this case $\iii'\{s'',t'\}\preceq'
\iii'\{s',t'\}$.

Let $v=\iii'\{s'',t'\}$.
If $v\in V$, then $v\prec' s''$
and $h'_{\nu, \mu}(s')=s''$ imply $v\prec' s'$ by (R2)(a).
Thus $v\preceq' t',s'$, and so  $v\preceq'\iii'\{s',t'\}$.

Now assume that $v\in X'_\nu\cup X'_\mu$. If $v\in \bar
Y=X'_\nu\cap X'_\mu$, then $v\prec' s'$, so
$v\prec'\iii'\{s',t'\}$.  We show that it is not possible that
$v\notin \bar Y$. For this, assume that $v\in (X'_\nu \cup
X'_\mu)\setm \bar Y$. Without loss of generality, we may suppose
that $v\in X'_\nu\setm X'_\mu$. Then, by (R2)(d), there is a $w\in
\bar Y$ such that $v\prec' w\prec' s''$. Thus $v=\iii
'\{w,t'\}=\iii'_\nu\{w,t'\}\in  \bar Y$ by Lemma
\ref{lm:kerneldown}.

Moreover, $\{s,t\}\in \br X_\nu;2;$ and
$\iii\{s,t\}=\iii'\{s',t'\}=\iii_\nu\{s,t\}$ because
$g_\nu(s')=h(s')=s$ and $g_\nu(t')=h(t')=t$.

\medskip\noindent{\em Case 3.   $t\in X_\mu\setm {\bar Z}$}.

Proceeding as in Case 2, we can show that $\iii'\{s',t'\}\preceq'
\iii'\{s'',t'\}=\iii_\mu\{s,t\}$.

\smallskip

Finally, assume that  $t\in {\bar Z}$. Then $h^{-1}(t)=\{t',t''\}$
for some $t'\in \bar Z'_{\nu}$ and   $t''\in \bar Z'_\mu $.

Note that by Cases (2) and (3),
\begin{equation}\notag
\iii'\{s'',t'\}\preceq' \iii'\{s',t'\} \text{ and }
\iii'\{s',t''\}\preceq' \iii'\{s'',t''\}.
\end{equation}
Since $\iii'\{s',t'\}=\iii_{\nu}\{s,t\} =
\iii_{\mu}\{s,t\}= \iii'\{s'',t''\}$ by the construction of $r'_\nu$ and $r'_\mu$,
we have
\begin{equation}\label{eq:max2}
\iii'\{s',t'\}=\iii'\{s'',t''\}=\max_{\prec'}(\iii'\{s',t'\},\iii'\{s'',t'\} ,
\iii'\{s',t''\},\iii'\{s'',t''\} ).
\end{equation}
Moreover, in this case $\{s,t\}\in \br X_\nu;2;\cap \br X_\mu;2;$
and we have just proved that
$\iii\{s,t\}=\iii_\nu\{s,t\}=\iii_\mu\{s,t\}$.

\end{proof}

By Claim \ref{cl:wd} above, $r$ is well-defined.
%
Since $\iii\supseteq \iii_\nu\cup \iii_\mu$,
it is easy to check that if $r\in P$ then $r\leq r_{\nu},r_{\mu}$.
So, the following claim completes the verification of the  chain
condition.

\begin{claim}
$r\in P$.
\end{claim}

\begin{proof}[Proof]
(P1) and  (P2)  are clear.

\smallskip
\noindent (P3) Assume that $\{s,t\}\in [X]^2$. Without loss of
generality, we may assume that $s,t$ are compatible but not
comparable in $\<X,\preceq\>$. Note that by (\ref{eq:prec}), (\ref{eq:0}) and
condition (P3) for $r'$, we have $\iii\{s,t\}\prec s,t$. So, we
have to show that if $v\prec s, t$ then $v\preceq \iii\{s,t\}$.

Assume that $v\prec s,t$. Then, $v\in U$ and there are $s',t'\in
X'$ such that $h(s')=s$, $h(t')=t$ and $v\prec's',t'$. By (P3) for
$r'$, $v\preceq'\iii'\{s',t'\}$. Now as $v,\iii'\{s',t'\},
\iii\{s,t\}\in U$ and $h\restriction U = id$, we infer from
(\ref{eq:0}) that $v \preceq'\iii'\{s',t'\} \preceq'\iii\{s,t\}$
and hence $v\preceq \iii\{s,t\}$.

%
%
%
%
%
%

\medskip
\noindent (P4) Assume that $s,t\in X$ are compatible but not
comparable in $\<X,\preceq\>$. Let $v=\iii\{s,t\}$.

\smallskip
\noindent (a) In this case ${\pi}(s), {\pi}(t)<{\eta}.$ Then
$s,t\in X\setm (Z_{\nu}\cup Z_{\mu})=U$, so $h(s)=s$ and $h(t)=t$.
Thus $\iii\{s,t\}=\iii'\{s,t\}$. Hence, it follows from condition
(P4)(a) for $r'$ that  ${\pi}(\iii\{s,t\})\in o(s)\cap o(t)$.

\smallskip
\noindent (b) In this case ${\pi}(s)<{\eta}$ and
${\pi}(t)={\eta}$. Then $s\in X\setm (Z_{\nu}\cup Z_{\mu})=U$ and
$t\in Z_{\nu}\cup Z_{\mu}$.

By (\ref{eq:0}) and Claim \ref{cl:wd}, there is a $t^*\in
Z'_\nu\cup Z_\mu'$ such that $h(t^*)=t$ and
$\iii\{s,t\}=\iii'\{s,t^*\}$.

Now, applying (P4)(a) for $r'$, we infer that
$\pi(v)\in o(s)\cap o(t^*)$. Since $\pi(t^*)\in E$, we have
$o(t^*)\subseteq E$ by Claim \ref{cl:epsilon}. Then
we deduce that $\pi(v)\in o(s)\cap E$, which was to be proved.


\smallskip
\noindent (c) The same as (b).

\smallskip
\noindent (d)
 In this case ${\pi}(s)={\pi}(t)={\eta}$.
If $\{s,t\}\in \br Z_\nu;2;$ then $\iii\{s,t\}=\iii_\nu\{s,t\}$,
and
%
by (P4)(d) for $r_{\nu}$, we deduce that
$\pi(\iii\{s,t\})\in F\{\xi(s),\xi(t)\}$. A parallel argument
works if $s,t\in Z_{\mu}$.

So we can assume that $s\in Z_{\nu}\setm Z_\mu$ and $t\in
Z_{\mu}\setm Z_\nu$. Note that there are a unique $s'\in Z_\nu'$ with
$h(s') = s$ and a unique $t'\in Z_\mu'$ with $h(t') = t$. Then,
$v=\iii\{s,t\} = \iii'\{s', t'\}\in U$. Hence either $v\in V$, or
$v\in X_{\nu}\cup X_{\mu}$ and in this case there is a $w\in
X_{\nu}\cap X_{\mu}$ with $v\prec' w$ by (R2)(d).

In both cases ${\pi}(v)<{\gamma_0}$. Note that, applying (P4)(a)
in $r'$ for $s'$, $t'$ and $v=\iii'\{s',t'\}$, we obtain
${\pi}(v)\in o(s')\cap o(t')$. Since ${\pi}(s'), {\pi}(t')\in E$
we have $o(s')\cup o(t')\subseteq E$ by Claim \ref{cl:epsilon}.
Thus ${\pi}(v)\in  E $. And since ${\pi}(v)<{\gamma_0}$, we have
${\pi}(v)\in F\{\xi(s),\xi(t)\}\cap E$, which was to be proved.

\smallskip
\noindent (P5) Assume that $s,t\in X$, $s\prec t$ and
$\Lambda=J({\pi}(s), {\pi}(t))$ isolates $s$ from $t$. Then
$s\notin T_{\eta}$, so $h(s)=s$.

If $t\notin T_{\eta}$ then $h(t)=t$, so we are done because
$r'$ satisfies (P5).

Assume that  $t\in T_{\eta}$. As $s\prec t$, there is a $t'\in
T_{{\varepsilon}_{{\zeta}({\nu})}}\cup
T_{{\varepsilon}_{{\zeta}({\mu})}}$ such that $h(t') = t$ and
$s\prec' t'$. Since $\pi(t')\in E$,  by Claim \ref{cl:jplussz} we
have $J(\pi(s),\pi(t'))=I({\pi}(s),1)=J(\pi(s),\pi(t))$. Applying
(P5) in $r'$ for $s$ and $t'$, we obtain a $v \in X'$ such that
$s\prec' v \preceq 't'$ and $\pi(v') = \Lambda^+$. But as
$\zeta(\nu),\zeta(\mu)$ are limit ordinals, we have $v \prec' t'$,
and hence $v\in X'\setm (Z'_{\nu}\cup Z'_{\mu})=U$. Then $h(v)=v$,
so $s\prec v\prec t$, which was to be proved.
\end{proof}

Hence we have proved that $\pcal$ satisfies the $\ka^+$-chain
condition, which completes the proof of Theorem \ref{Theorem_2}.
\end{proof}

\section{Appendix}
We explain in detail how  Proposition 23 was proved in \cite{M}.

Assume that $Z\subseteq P_{\eta}$ is a separated set. Let $\bar X$
be the root of $\{X_p:p\in Z\}$. For every $n\in \omega$ and every
$I\in \ical_n$ with $\cf({I^+})= \kappa^+$, we define $\xi(I) =$
the least ordinal $\gamma$ such that ${\varepsilon}^I_{\gamma}
\supseteq \pi[\bar X]\cap I$ and we put $\gamma(I) =
{\varepsilon}^I_{\xi(I) + \kappa}$. Now for every $\alpha < \eta$,
if there is an $n < \omega$ and an interval $I\in \ical_n$ with
$\cf({I^+})= \kappa^+$ such that $\alpha \in I$ and $\gamma(I)\leq
\alpha$, we consider the least natural number $k$ with this
property and write $I(\alpha) = I(\alpha,k)$. Otherwise, we write
$I(\alpha) = \{\alpha\}$. Then we say that $Z$ is {\em pairwise
equivalent} iff for every $p,q\in Z$ and every $s\in X_p$,
$I(\pi(s)) = I(\pi(h_{p,q}(s)))$. In \cite{M}, the following two
lemmas were proved:


\begin{lemma}[{\cite[Lemma 2.5]{M}}]
Every set in $\br P_{\eta};{{\kappa}^+};$
has a pairwise equivalent subset of size
${\kappa}^+$.
\end{lemma}

\begin{lemma}[{\cite[Lemma 2.6]{M}}]
A pairwise equivalent set $Z\subseteq P_{\eta}$ of size
${\kappa}^+$ is linked.
\end{lemma}

To get Proposition 23 we explain that the proof of {\cite[Lemma
2.6]{M} actually gives the following statement:

\vspace{2mm} {\em If $Z\subseteq P_{\eta}$ is a pairwise
equivalent set of size ${\kappa}^+$, then there is an ordinal
${\gamma}<{\eta}$ such that every  $p,q\in Z$ have a common
extension  $r\in P_{\eta}$ satisfying (R1)--(R2).}

\vspace{2mm} As above, we denote by $\bar X$ the root of
$\{X_p:p\in Z \}$. Assume that $p,q\in Z$ with $p\neq q$. First
observe that the ordering $\prec_r$ is defined in \cite[Definition
2.4]{M}. For this, adequate bijections $g_1 : X_r\setm (X_p\cup
X_q)\to X_p\setm {\bar X}$ and $g_2 : X_r\setm (X_p\cup X_q)\to
X_q\setm {\bar X}$ are considered in such a way that
$g_2=h_{p,q}\circ g_1$. Then since $g_2=h_{p,q}\circ g_1$,
\cite[Definition 2.4]{M}(b) and (c) imply (R2)(a) and
\cite[Definition 2.4]{M}(d) and (f) imply (R2)(b). Also, (R2)(c)
follows directly from \cite[Definition 2.4]{M}(d) and (f), and
(R2)(d) is just \cite[Definition 2.4]{M}(e) and (g).
So, we have
verified (R2).

To check (R1), i.e. to get the right ${\gamma}$ we need a bit more
work. Let
\begin{equation}
\mc J=\{I(\pi(s)):s\in X_p\}
\end{equation}
where $p\in Z$. Since $Z$ is pairwise equivalent, $\mc J$ does not
depend on the choice of $p\in Z$. For every $I\in {\mathbb
I}_{\eta}$ with $\cf(I^+)={\kappa}^+$ we can choose a set $D(I)\in
[E(I)\cap \gamma(I)]^{\kappa}$ unbounded in $\gamma(I)$.  We claim
that
\begin{equation}
{\gamma}=\sup(\bigcup\{D(I): I\in \mc J\})+1
\end{equation}
works.

First observe that ${\gamma}<{\eta}$, because
$\cf({\eta})={\kappa}^+$,  $|\mc J| < {\kappa}$ and $|D(I)| =
{\kappa}$ for any $I\in \mc J$.

Now assume that $p,q\in Z$ with $p\neq q$. Write $L_p = \pi[X_p]$,
$L_q = \pi[X_q]$ and ${\bar L} = \pi[{\bar X}]$. Let
$\{\alpha_{\xi}:\xi < \delta \}$ and $\{\alpha'_{\xi}:\xi < \delta
\}$  be the strictly increasing enumerations of $L_p\setminus \bar
L$ and $L_q\setminus \bar L$ respectively. In the proof of
\cite[Lemma 2.6]{M}, for each $\xi < \delta$ an element
$\beta_{\xi}\in D(I(\alpha_{\xi})) = D(I(\alpha'_{\xi}))$ was
chosen, and then a condition $r\leq_{\eta} p,q$ was constructed in
such a way that $X_r = X_p\cup X_q\cup Y$ where $Y\cap (X_p\cup
X_q) = \emptyset$ and $\pi[Y] = \{\beta_{\xi} : \xi <\delta \}$.
Then since $\{\beta_{\xi}: \xi < \delta \}\subseteq
\bigcup\{D(I):I\in \mc J \}$, we infer that

\begin{equation}
 \sup {\pi} [X_r\setm (X_p\cup X_q)] = \sup {\pi}[Y]<{\gamma},
\end{equation}
which was to be proved.

\end{document}